\documentclass[12pt]{amsart}
\usepackage[utf8]{inputenc}
\usepackage[T1]{fontenc}
\usepackage[greek,english]{babel}
\usepackage{epsfig}
\usepackage{float}
\usepackage{enumitem}

\usepackage{amsfonts,amsmath,amssymb,amsthm,hyperref,setspace,verbatim,times,longtable}
\usepackage{graphicx, complexity}
\usepackage{MnSymbol}
\usepackage{upquote}
\usepackage[a4paper,margin=1.9cm,top=2.5cm,bottom=2.5cm,centering,vcentering]{geometry}

\newcommand{\abs}[1]{\lvert#1\rvert}

\DeclareMathOperator{\ad}{AD}
\DeclareMathOperator{\bd}{BD}
\usepackage{graphicx}

\makeatletter
\newcommand*\bigcdot{\mathpalette\bigcdot@{1}}
\newcommand*\bigcdot@[2]{\mathbin{\vcenter{\hbox{\scalebox{#2}{$\m@th#1\bullet$}}}}}
\makeatother

\usepackage{color}
\definecolor{ao(english)}{rgb}{0.0, 0.0, 1.0}
\hypersetup{colorlinks=true, linkcolor=ao(english),citecolor=ao(english)}

\usepackage[normalem]{ulem}

\usepackage{color}

\newtheorem{theorem}{Theorem}

\newtheorem{lemma}{Lemma}
\newtheorem{remark}[theorem]{Remark}

\title[A Novel Approach to Counting Perfect Matchings of Graphs]{A Novel Approach to Counting Perfect Matchings of Graphs}

\author[P. Paul]{Pravakar Paul}
\address{Mathematical and Physical Sciences division, School of Arts and Sciences, Ahmedabad University, Ahmedabad 380009, Gujarat, India}
\email{pravakar.paul@ahduni.edu.in, manjil@saikia.in}

\author[M. P. Saikia]{Manjil P. Saikia}



\begin{document}

\begin{abstract}
We build a new perspective to count perfect matchings of a given graph. This idea is motivated by a construction on the relative cohomology group of surfaces. As an application of our theory, we reprove the celebrated Aztec Diamond theorem, and show how alternating sign matrices naturally arises through this framework.
\end{abstract}

\maketitle

\section{Introduction}

For a graph $G=(V,E)$, there are two problems of interest that have eluded geometers, enumerative combinatorialists, graph theorists, physicists, and theoretical computer scientists for the past century:
\begin{enumerate}[label=(\Alph*)]
    \item Does there exist a perfect matching for the given graph $G$ ? More generally, study the invariants of $G$ that leads to the obstruction to perfect matching.
    \item If the answer to the first question is affirmative, can we enumerate the number of perfect matchings of $G$? 
\end{enumerate}

Given a graph $G=(V,E)$, a perfect matching in $G$ is a subset $E^\prime$ of the edge set $E$, such that every vertex in the vertex set $V$ is adjacent to exactly one edge in $E^\prime$. The first problem is topological in nature, whereas the second problem is combinatorial.  Pioneering work in these directions have been done by many mathematicians, in particular by Tutte \cite{Tutte}, Kasteleyn \cite{Kasteleyn}, Temperley \& Fisher \cite{temp}, Thurston \cite{Thurston}, Conway \& Lagarias \cite{ConwayLagarias}, etc. For related work in different directions, we refer the reader to the books by Lov{\'a}sz \& Plummer \cite{LovaszPlummer}, and by Lucchesi \& Murty \cite{LucchesiMurty}. For a short survey, one can refer to \cite[Sections 1--3]{Thomas}.

In general, counting the number of perfect matchings of a graph is a difficult problem; for instance, Valiant \cite{Valiant} showed that there is no polynomial time algorithm for counting perfect matchings in general graphs (even for bipartite graphs), unless $\P=\NP$. Nonetheless, over the last several decades, mathematicians have studied this problem for specific graphs $G$ and have come up with several techniques to find the number of perfect matchings of $G$ (most of which are bipartite graphs). In this paper, we continue this study and build a new perspective to count perfect matchings of a graph (not necessarily bipartite).

The motivation for this work originated from the following construction in algebraic topology. 
\begin{figure}[H]
    \centering
    \includegraphics[width=0.55\linewidth]{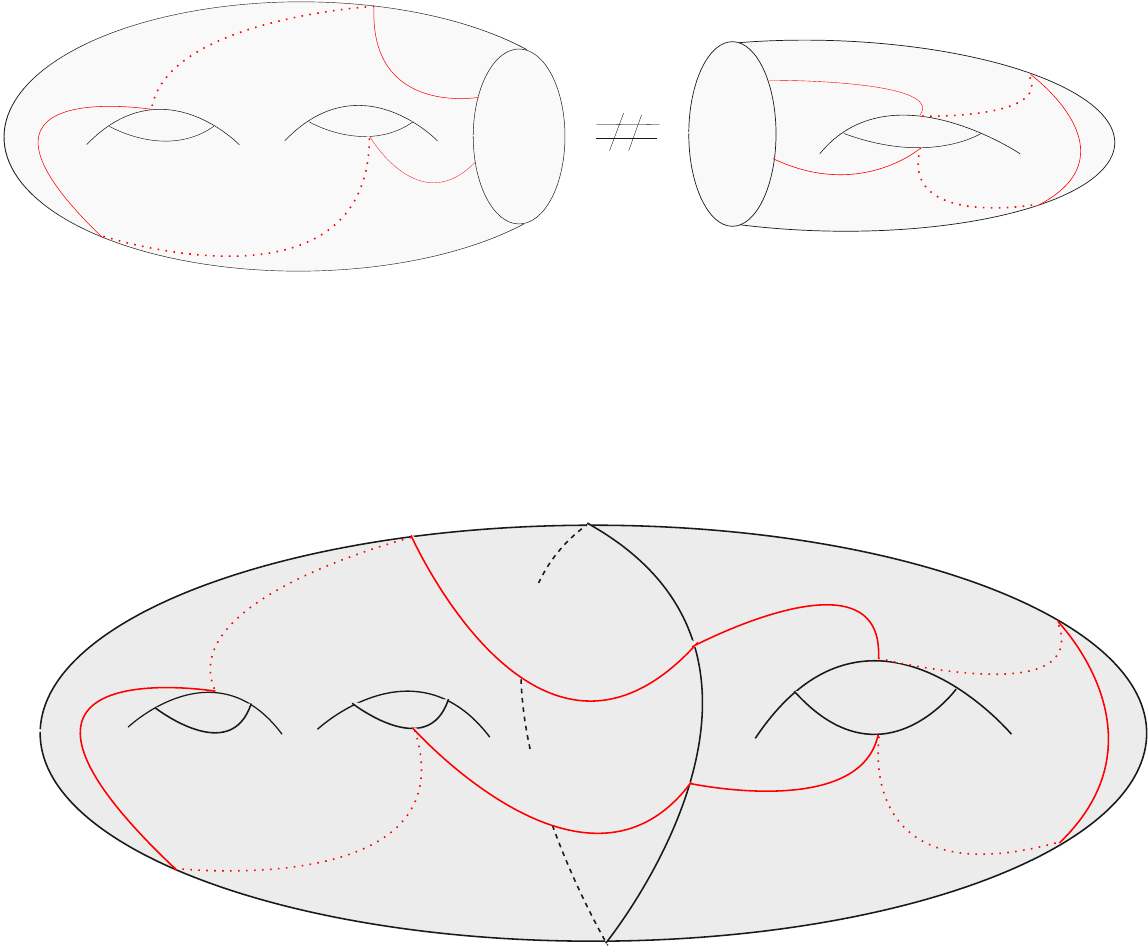}
    \caption{The pairing map $\varphi$ on relative homology group.}
    \label{fig:1}
\end{figure}
Assume $ \left( S^{g_1}, \partial S^{g_1}  \right)$ and $ \left( S^{g_2}, \partial S^{g_2}  \right)$ are two surfaces of genus $g_1$ and $g_2$ with boundaries  respectively. Additionally, let us assume 
 \[ \partial S^{g_1} \cong \partial S^{g_2} \cong S^{1}  .\]
 We can construct a closed surface $\Sigma^{g_1+g_2}$ of genus $ \left( g_1+g_2 \right)$ by identifying their respective boundary components through a homeomorphism $ f: \partial S^{g_1} \to \partial S^{g_2} $. In this situation, there exists a canonical pairing map 
\[ \varphi: H_{1}  \left( S^{g_1}, \partial S^{g_1} \right) \otimes H_{1} \left( S^{g_2}, \partial S^{g_2} \right)  \longrightarrow H_{1} \left( \Sigma^{g_1+g_2} \right), \]
where $H_{1}  \left( S^{g_i}, \partial S^{g_i} \right)$ denotes the relative homology group relative to the boundary $\partial S^{g_i}$ and $ H_{1} \left( \Sigma^{g_1+g_2} \right)$ denotes the absolute homology group. The map is defined by gluing the respective representatives of the homology classes boundary-wise. For $g_1=2$ and $g_2=1$ see Figure \ref{fig:1}. We think of a relative homology class in the pairing map $\varphi$ to correspond to a perfect matching of the graph $G=(V,E)$ with a \textit{defect}, that is, a collection of missing vertices, and a absolute homology class as a perfect  matching of $G$ (see Theorem \ref{thm:main} for the analogy). In this paper we have essentially \textit{combinatorified} the map $\varphi$.

Assume $G_1=(V_1,E_1)$ and $G_2=(V_2, E_2)$ are two graphs with a choice of $n$-many distinguished vertices of $V_i$ for $i=1, 2$. Call them $\{x_1, \ldots, x_n \} \subset V_1$ and $\{y_1, \ldots, y_n \} \subset V_2$ respectively. We define the \textit{connected sum} along these distinguished vertices as   \[ G_1 \# G_2 := G_1 \cupdot G_2 / x_i \sim y_i , \forall 1 \leq i \leq n \] 
where $\cupdot$ denotes the disjoint union.  
The vertex set of $G_1 \# G_2$ is defined as $V_1 \cupdot V_2 / x_i \sim y_i$, for all $1 \leq i \leq n $ and the edge set as $E_1 \cupdot E_2 $. 
For any graph $G$, let $M(G)$ denote the number of perfect matchings of $G$. The goal is to understand $M(G_1 \# G_2)$ in terms of $M(G_1)$ and $M(G_2)$.

We define an algebra structure that captures the behavior of  $M(G_1 \# G_2)$ along the boundary. 
Define the \textit{matching algebra} $\mathcal{M}$ over $\mathbb{Z}$ with two generators $y$ and $n$ given by the following relations
\[ \mathcal{M} := \mathbb{Z} \left<y,n \right>/ \left<   yn=yn=n; n^2=0; y^2=y \right>\]
where $G(\epsilon_1, \cdots , \epsilon_{n}) \geq 0$. Here the variable $y$ indicates \textit{yes} or \textit{presence} and the variable $n$ indicates \textit{no} or \textit{absence}. For a graph $G$ with a choice of $n$ distinguished vertices $\{x_1, \ldots, x_n \} \subset V $ we shall assign an element $v_{G} \in \mathcal{M}^{\otimes n} $ of the form \[ v_{G}= \sum_{\epsilon_{i} \in \{ y,n \}} G(\epsilon_1, \ldots ,\epsilon_n) \epsilon_1 \otimes \cdots \otimes \epsilon_n. \] Define an involution $ \bar{\epsilon}$ which switches $y$ to $n$ and $n$ to $y$. That is, as a set we have \[ \{\epsilon, \bar{\epsilon} \}= \{ y,n\}. \] Our main result is this paper is the following theorem.
\begin{theorem}\label{thm:main}
    For graphs $G_{1}= (V_1 , E_1)$ and $G_2=(V_2, E_2)$ with a choice of distinguished vertices as mentioned above, let $v_{G_i}$ denote the element defined as in the previous paragraph. Then  
    \[ M(G_1 \# G_2)= \sum G_{1}(\epsilon_1, \ldots, \epsilon_n ) G_{2}(\bar{\epsilon_1}, \ldots,  \bar{\epsilon_n}). \]
\end{theorem}

As an application of Theorem \ref{thm:main} we prove the celebrated Aztec diamond theorem later.

This paper is organized as follows: in Section \ref{sec:proof} we prove Theorem \ref{thm:main} and related lemmas. We also provide some simple applications which we use in the later sections of the paper. In Section \ref{sec:ad} we explain two objects of interest, namely Aztec Diamonds and alternating sign matrices. In Section \ref{sec:ad1} we give a proof of one version of the celebrated Aztec Diamond Theorem (to be explained in Section \ref{sec:ad}) using Theorem \ref{thm:main}. In Section \ref{sec:ad2} we give a proof of another version of the Aztec Diamond Theorem using Theorem \ref{thm:main}. We finally close the paper with some concluding remarks in Section \ref{sec:conc}. Throughout this paper we shall use red color to draw the distingushed vertices. 

\section{Proof of Theorem \ref{thm:main} and Some Simple Applications}\label{sec:proof}

For any graph $G= (V,E)$ with $n$ distinguished vertices $\{v_1, \ldots, v_n \} \subset V$ we associate a non-negative integer $G(\epsilon_1, \ldots, \epsilon_n)$ defined as the number of perfect matchings of the subgraph  of $G$ where we delete the $i$-th vertex if $\epsilon_i = y$ or we keep it if $\epsilon_i= n$. Note that we have the following identity
 \[G(n,\ldots, n)= M(G). \]
 Said differently, $\epsilon_i=y$ indicates the situation where the identified vertex $v_i$ is available to match with a vertex in the other graph in the connected sum. Whereas, $\epsilon_i=n$ indicates the situation where the identified vertex $v_i$ is already matched with some other vertex in $G$.  With this understanding we associate the element $v_{G} \in \mathcal{M}^{\otimes n}$ which is a sum of $2^{n}$ terms. We call this the \textit{state sum expansion} of $G$ associated to $\{v_1, \cdots, v_n \}$. Next we prove Theorem \ref{thm:main}.
\begin{proof}[Proof of Theorem \ref{thm:main}]
    If $X$ is a perfect matching of $G_1 \# G_2$, then $X$ matches the identified vertex $x_i=y_i$ either to a vertex in $G_1$ or to a vertex in $G_2$, for all $i$. If it matches with a vertex in $G_1$ then it must not be available for any vertex in $G_2$ to match with. Thus $\epsilon_i=n$ and $ \bar{\epsilon} = y$ in this case. On the other hand, if $X$ matches it with a vertex in $G_2$ then it must not be available for any vertex in $G_1$ to match with. In this case  $\epsilon_i=y$. In the case of a perfect matching $X$ notice that we must have $ \{ \epsilon_i, \bar{\epsilon_i} \}= \{y, n \}$. Conversely, we observe that any perfect matching of $G_1 \# G_2$ arrives in this way. For different choices of $\epsilon_1 \otimes \ldots \otimes \epsilon_n$ yield different perfect matchings of $G_1 \# G_2$. 
\end{proof}

\noindent Although the proof of the main theorem is deceptively simple, we can still get quite a lot out of it. In the next series of lemmas we show some simple applications of Theorem \ref{thm:main} with some explicit calculations.
\begin{lemma}\label{lem1}
    The state sum expansion of the graph $G$ with two distinguished vertices $\{v_1, v_2\}$ is given by 
    \begin{equation}\label{eq:0}
        y \otimes n + n \otimes y,
    \end{equation} where $v_1$ and $v_2$ are the bottom vertices from left to right of Figure \ref{fig:lem1}.
\end{lemma}

\begin{figure}[!htb]
    \centering
    \includegraphics[width=0.15\linewidth]{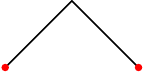}
    \caption{The graph $G$ in Lemma \ref{lem1}.}
    \label{fig:lem1}
\end{figure}

\begin{proof}
    Let $w$ denote the top vertex. The graph $G$ is connected along $v_1$ and $v_2$. The top vertex $w$ must be matched with either $v_1$ or $v_2$. If $w$ is matched with $v_1$, then $v_1$ can not contribute in the matching $M$. Also, the other vertex $v_2$ must be present in $M$. Thus $n\otimes y$ is the contribution from this matching. Reversing the role of $v_1$ and $v_2$ we get the other term.
\end{proof}


\begin{lemma}\label{lem2}
    The state sum expansion of the graph $G$ with three distinguished vertices $\{v_1, v_2, v_3\}$ is given by \begin{equation}\label{eq:1}
        y\otimes n \otimes n + n\otimes y \otimes n + n\otimes n \otimes y,
    \end{equation} where $v_1, v_2$ and $v_3$ are the bottom vertices from left to right of Figure \ref{fig:lem2}.
\end{lemma}

\begin{figure}[!htb]
    \centering
    \includegraphics[width=0.35\linewidth]{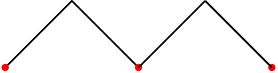}
    \caption{The graph $G$ in Lemma \ref{lem2}.}
    \label{fig:lem2}
\end{figure}

\begin{proof}
    Let $w_1$ and $w_2$ denote the top vertices from left to right. As $w_i's$ do not contribute in the connected sum, they must be matched with $v_j's$. There are three possibilities:
    \begin{itemize}
        \item  $w_1$ matches with $v_1$ and $w_2$ matches with $v_2$ contributing the term $n \otimes n \otimes y$,
        \item  $w_1$ matches with $v_1$ and $w_2$ matches with $v_3$ contributing the term $n \otimes y \otimes n$, and
        \item $w_1$ matches with $v_2$ and $w_2$ matches with $v_3$ contributing the term $y \otimes n \otimes n$. 
    \end{itemize}
\end{proof}

\begin{remark}\label{rem:2}
    The expression \eqref{eq:1} is same as the product $(y \otimes n + n \otimes y)\cdot (y \otimes n + n \otimes y)$ (that is, multiplying two copies of the expression \eqref{eq:0}) where we define \[(\epsilon_1 \otimes \epsilon_2)\cdot (\epsilon_3 \otimes \epsilon_4) := \epsilon_1 \otimes (\epsilon_2\cdot  \epsilon_3) \otimes \epsilon_4.\]
If we perform the multiplication, then by distributive law we get \[ (y \otimes n + n \otimes y)\cdot (y \otimes n + n \otimes y)= y \otimes (n\cdot y) \otimes n+ y \otimes (n\cdot n) \otimes y + n \otimes (y\cdot y) \otimes n + n \otimes (y\cdot n) \otimes y.\]
This is the same expression as \eqref{eq:1}, as the middle term is $0$ since $n\cdot n=0$ in $\mathcal{M}$.
\end{remark}

The calculation in Remark \ref{rem:2} is not a coincidence. The next lemma explains the phenomenon. Before we state it properly, we set up few notations. Let $(G_1; v_1, \ldots, v_{i+j})$ and $(G_2; w_1,  \ldots, w_{j+k})$ be two graphs with $ \left(i+j \right)$ and $ \left( j+k \right)$ distinguished vertices respectively. Assume $v_{G_1}$ and $v_{G_2}$ denote the state sum decomposition of $G_1$ and $G_2$ respectively. We construct a new graph $G$ with $ \left(i+j+k \right)$ distinguished vertices defined as follows: 
\begin{equation}\label{eq:cc}
G := G_1 \cupdot G_2 / v_{i+1} \sim w_1, \cdots , v_{i+j} \sim w_j.
\end{equation}

We want to express $v_{G}$ in terms of $v_{G_{1}}$ and $v_{G_{2}}$. To understand it concretely, we need the notion of internal multiplication. Let $I \subset [n]$ and $J \subset [m]$ with $|I|=|J|$ where we denote the set $\{1,2,\ldots, k\}$ by $[k]$. We can define an \textit{internal multiplication} $\phi_{I,J}$ as follows: 
\[ \phi_{I,J}: \mathcal{M}^{\otimes n} \otimes \mathcal{M}^{\otimes m} \to \mathcal{M}^{\otimes (n+m- |I|)},\] 
where the internal multiplication occurs only on the coordinates of $I$ and $J$. More explicitly, if \[|I|=|J|=k, \quad I = \{ i_1 < \cdots < i_k \}, \quad \text{and} \quad J= \{ j_1 < \cdots < j_k \}.\] Take two simple tensors $ \left( A \cdot \epsilon_{1} \otimes \cdots \otimes \epsilon_{n} \right) \in \mathcal{M}^{\otimes n}$ and $\left( B \cdot \eta_{1} \otimes \cdots \otimes \eta_{m}\right) \in \mathcal{M}^{\otimes m}$, where $\epsilon_i, \eta_{j} \in \{y, n\}$. Then,
\begin{multline*}
    \phi_{I,J} \left( (A \cdot \epsilon_1 \otimes \cdots \otimes \epsilon_n) \otimes (B \cdot \eta_1 \otimes \cdots \otimes \eta_{m}) \right)\\:= \left( A \cdot B \right) \epsilon_1 \otimes \cdots \otimes \epsilon_{i_1-1}\otimes (\epsilon_{i_1} \cdot \eta_{j_1}) \otimes \cdots \otimes (\epsilon_{i_{k}} \cdot \eta_{j_{k}}) \otimes \cdots \otimes \epsilon_{n} \otimes \eta_1 \otimes \cdots \otimes \widehat{\eta_{j_1}} \otimes \cdots \otimes \widehat{\eta_{j_k}} \otimes \cdots \otimes \eta_{m},
\end{multline*}

\noindent where~~$\widehat{\cdot}$~~denotes the absence of the variable. For general elements $W \in \mathcal{M}^{\otimes n}$ and $V \in \mathcal{M}^{\otimes m}$ , we extend the definition of internal multiplication 
\(\phi_{I,J} ( W \otimes W ) \) by distributivity. We give an example of the internal multiplication:  for $n=3 , m=4$ and  for $I= \{ 2,3 \},  J=\{ 1,4 \}$:
\[ \phi_{I,J} \left( ( \epsilon_1 \otimes \epsilon_2 \otimes \epsilon_3) \otimes (\eta_1 \otimes \eta_2 \otimes \eta_3 \otimes \eta_4) \right) = \epsilon_1 \otimes (\epsilon_2 \cdot \eta_1) \otimes (\epsilon_3 \cdot \eta_4) \otimes \eta_2 \otimes \eta_3. \]
If $I,J$ are well understood we shall remove them from $\phi$.

\begin{lemma}[The Patching Lemma]\label{lem3}
For $I=\{i+1, \ldots, i+j \} \subset [i+j]$ and $J= \{1, \ldots, j \} \subset [j+k]$, the state sum decomposition of the graph $G$, where $G= (V,E)$ with $n$ distinguished vertices $\{v_1, \ldots, v_n \} \subset V$ is given by  \[v_{G}= \phi_{I,J} (v_{G_1} \otimes v_{G_2}),\] where $G_1$ and $G_2$ are as described in \eqref{eq:cc}.  
\end{lemma}
\begin{proof}
In a connected sum between two graph $G_1$ and $G_2$, if $v \in G_1$ gets identified with $w \in G_2$, then in the disjoint union of $G_1$ and $G_2$ there could be four possible situations in a matching $M$:
\begin{itemize}
    \item \textbf{Case $1$}: $v$ is matched by an edge in $G_1$ and $w$ does not get matched,
    \item \textbf{Case $2$}: $v$ does not get matched and $w$ is matched by an edge in $G_2$,
    \item \textbf{Case $3$}: both $v$ and $w$ are matched in $G_1$ and $G_2$ respectively, and
    \item \textbf{Case $4$}: neither of $v$ and $w$ are matched.
\end{itemize}

In the connected sum of $G_1 \# G_2$, the situations in case $1$ and $2$ are reflected by the multiplication \[yn=ny=n\] in the matching algebra $\mathcal{M}$ along the respective coordinates of $v_{G_1}$ and $v_{G_2}$. However, case $3$ can never arise in a matching of $G_1 \# G_2$ which is reflected in the algebra multiplication $n^2=0$. Lastly, if neither $v$ and $w$ get picked up in the matching $M$, they shall stay alive in $G_1 \# G_2$ which is reflected in the algebra multiplication $y^{2}=y$.     
\end{proof}

\section{Aztec Diamonds \& Alternating Sign Matrices}\label{sec:ad}

In 1992, Elkies, Kuperberg, Larsen, and Propp \cite{AD1, AD2} introduced a new class of objects which they called Aztec diamonds. The Aztec diamond of order $n$ (denoted by $\ad(n)$) is the union 
of all unit squares inside the contour $\abs{x}+\abs{y}=n+1$ (see the top left of Figure \ref{fig:ad} for an Aztec diamond of order $4$; at this moment we ignore the other figures). They considered the problem of counting the number of domino tilings of an Aztec diamond of order $n$ (see the top right of Figure \ref{fig:ad} for an example of such a tiling for $n=4$) and proved that this number is equal to $2^{n(n+1)/2}$. They gave four different proofs of this result, henceforth referred to as the Aztec Diamond Theorem.
\begin{theorem}[Aztec Diamond Theorem, \cite{AD1, AD2}]\label{adm}
The number of domino tilings of the Aztec diamond of order $n$ is $2^{n(n+1)/2}$.
\end{theorem}

\begin{figure}[!htb]
\centering
\includegraphics[scale=.99]{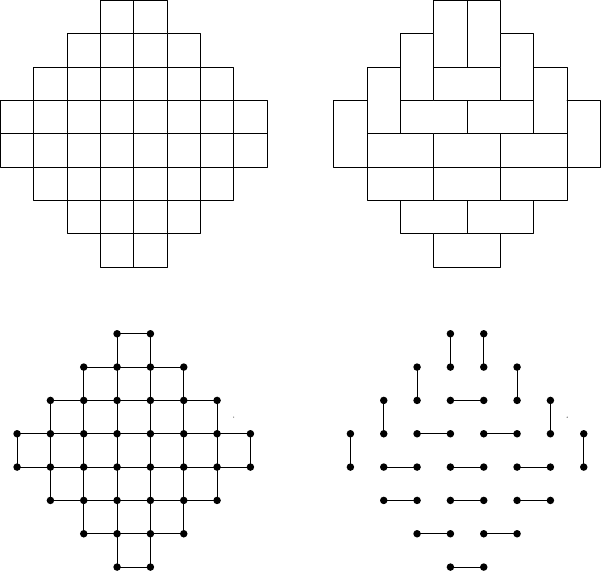}
\caption{Top row: the Aztec Diamond $\ad(4)$ (left) and a tiling of $\ad(4)$ (right); bottom row: the planar dual graph of $\ad(4)$ (left) and the perfect matching of the planar dual graph of $\ad(4)$ corresponding to the tiling directly above it (right).}
\label{fig:ad}
\end{figure}

We can recast the problem of enumerating domino tilings into a perfect matching problem. A perfect matching of a graph is a matching in which every vertex of the graph is incident to exactly one edge of the matching. It is easy to see that domino tilings of a region can be identified with perfect matchings of its planar dual graph, the graph that is obtained if we identify each unit square with a vertex and unit squares sharing an edge is identified with an edge. See the bottom left of Figure \ref{fig:ad} for the planar dual graph of an Aztec Diamond of order $4$, and the bottom right of Figure \ref{fig:ad} for the perfect matching corresponding to the tiling shown in the top right of Figure \ref{fig:ad}.

There are several techniques that combinatorialists have used over the years to calculate the number of perfect matchings of planar dual graphs of regions whose tiling enumeration number is required. Among these the technique found by Kuo \cite{kuo}, and later generalized by Ciucu \cite{ck} and the second author \cite{Saikia} is widely used. Previous proofs of Theorem \ref{adm} include the original proofs by Elkies, Kuperberg, Larsen and Propp \cite{AD1, AD2}; combinatorial proofs by Eu and Fu \cite{eufu}, Bosio and Van Leeuwen \cite{bosio}, and by Fendler and Grieser \cite{fendler}. In the next two sections we will present a proof of Theorem \ref{adm} which uses Theorem \ref{thm:main}, and the ideas of our matching algebra.

 The way we accomplish this goal is to exploit a connection between domino tilings of Aztec Diamonds with another class of combinatorial objects, called \textit{alternating sign matrices} (ASMs). An ASM of order $n$ is an $n\times n$ matrix where the entries come from the set $\{0,1,-1\}$, the row and column sums all equal $1$ and the non-zero entries alternate in sign along each row and column. An example of an ASM of order $7$ is

\[
\begin{pmatrix}
0&0&0&1&0&0&0\\0&1&0&-1&0&1&0\\1&-1&0&1&0&-1&1\\0&0&1&-1&1&0&0\\0&1&-1&1&-1&1&0\\0&0&1&-1&1&0&0\\0&0&0&1&0&0&0
\end{pmatrix}.
\] These matrices, first introduced by Robbins and Rumsey, in the 1980s, have given rise to a lot of nice enumerative conjectures and results. We refer the reader to the book by Bressoud \cite{bressoud} and the papers by Fischer \& Saikia \cite{FischerSaikia} and Behrend, Fischer \& Koutschan \cite{BehrendFischerKoutschan} for surveys of known and conjectured results related to ASMs.

If we rotate the planar dual graph of an Aztec Diamond by $45^\circ$, we see that this graph is made of $n$ rows of $n$ diamond shaped cells. If we assign an entry $1$, $0$ or $-1$ to each such cell in the perfect matching where a cell is covered by $2$, $1$ or $0$ edge(s) then we get a correspondence between a perfect matching and an ASM. This is illustrated in Figure \ref{fig:asm-domino} for the matrix
$$
\begin{pmatrix}
0&1&0\\
1&-1&1\\
0&1&0
\end{pmatrix}.
$$
Here the heavier lines are the edges that count towards the perfect matching.

\begin{figure}[!htb]
\begin{center}
    \scalebox{0.9}{\includegraphics{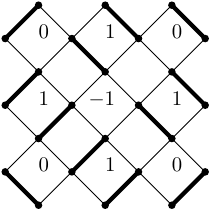}}
\caption{ASMs and Aztec diamonds.}
\label{fig:asm-domino}
\end{center}
\end{figure}

This gives us the following result which connects domino tilings of Aztec diamonds and enumeration of ASMs (see the work of Elkies, Kuperberg, Larson and Propp \cite{AD1,AD2}, as well as of Mihai Ciucu \cite{CiucuSymmetry}).
\begin{theorem}\label{thm:2-enum}
    The number of domino tilings of an Aztec Diamond of order $n$ is given by \[
    \sum_{A\in \mathcal{A}_{n}}2^{N_{+}(A)}=\sum_{A\in \mathcal{A}_{n+1}}2^{N_{-}(A)},\] where $\mathcal{A}_n$ is the set of all $n\times n$ ASMs and $N_{\pm}(A)$ is the number of $\pm 1$'s in $A$.
\end{theorem}
The sum above is also called $2$-enumeration of ASMs. It is not difficult to arrive at Theorem \ref{adm} from Theorem \ref{thm:2-enum} by observing that any ASM of order $n$ has $n$ more $1$'s than $-1$'s, which implies the following recurrence relation
\[
M(AD(n))=2^{n}M(AD(n-1)), \quad n\geq 2.
\]
From this and the starting values of $n$ we arrive at Theorem \ref{adm}.
\begin{remark}
    Initially the connection of Aztec diamonds and alternating sign matrices was established through an intermediate object called \textit{height function} (see. \cite{AD1}). This mysterious height function also appeared in the work of Thurston (see. \cite{Thurston}) in the context of tiling theory. Our construction does not rely on this intermediate object.  
\end{remark}

\section{Matching Algebra and the Aztec Diamond Theorem}\label{sec:ad1}

In this section we use the theories developed in Section \ref{sec:proof} to provide an algebraic explanation of the connection between domino tilings of an Aztec Diamond theorem with the $2$-enumeration of ASMs as mentioned in Theorem \ref{thm:2-enum}. The section will culminate in an independent proof of Theorem \ref{thm:2-enum} which will in turn give a proof of Theorem \ref{adm}.

\begin{lemma}\label{lem44}
The state sum decomposition of the $X$ shaped graph shown in Figure \ref{fig:lem4} with the distinguished vertices marked with the highlighted red dots, is given by \[n \otimes y \otimes y \otimes y + y \otimes n \otimes y \otimes y + y \otimes y \otimes n \otimes y + y \otimes y \otimes y \otimes n .\]
\end{lemma}

\begin{figure}[!htb]
    \centering
    \includegraphics[width=0.13\linewidth]{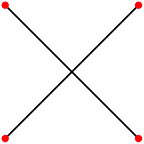}
    \caption{The $X$ shaped graph in Lemma \ref{lem44}.}
    \label{fig:lem4}
\end{figure}

\begin{proof}
    As the graph has one extra vertex that does not participate in the connected sum operation. This vertex must be paired with one of the other four remaining vertices. Once it is paired, the rest of the three vertices must be free. Thus we have the expression.
\end{proof}
For ease in calculations, in the case above we represent the state sum decomposition as the formal sum of matrices. That is,
\[ v_{G} = \begin{bmatrix}
    n & y \\
    y & y \\ 
\end{bmatrix} + 
\begin{bmatrix}
    y & n \\
    y & y \\ 
\end{bmatrix} + 
\begin{bmatrix}
    y & y \\
    y & n 
\end{bmatrix} + 
\begin{bmatrix}
    y & y \\
    n & y 
\end{bmatrix}.
\]
Note that we can not add these matrices like in linear algebra. These matrices are the formal representation of the tensor sum described above where the position of the entry keeps track of the position of the distinguished vertex. We introduce the following set of notations, which we use henceforth in this paper:
\[ X= \begin{bmatrix}
    n \\
    y
\end{bmatrix} + \begin{bmatrix}
    y \\
    n
\end{bmatrix} , \quad Y = \begin{bmatrix}
    y \\
    y
\end{bmatrix}, \quad Z = \begin{bmatrix}
    n \\
    n
\end{bmatrix}. \]
With this new set of notations, we see from above \[v_{G}= X \otimes Y + Y \otimes X  .\] The coordinate-wise multiplication on $X, Y$ and $Z$ satisfy the following set of relations:
\[ X^{2}= 2Z, \quad Y^{2}=Y, \quad XY= YX = X , \quad XZ= ZX=0. \]

\begin{figure}[!htb]
\begin{center}
    \scalebox{0.85}{\includegraphics{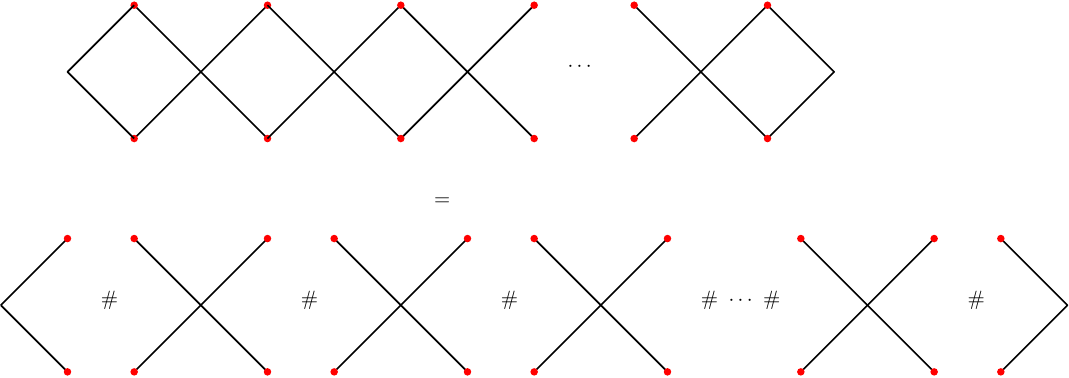}}
\caption{The graph $G_n$.}
\label{fig:pp8-1}
\end{center}
\end{figure}

In the next lemma we compute the state sum decomposition of the graph $G_{n}$ shown in Figure \ref{fig:pp8-1}. $G_n$ is the connected sum of $n$ different $X$'s and is capped of by the graph in Lemma \ref{lem1} on both the corners. This means, we need to calculate the expression \[ X (X \otimes Y + Y \otimes X) ^{n} X.\] This can now be achieved via repeated applications of Lemma \ref{lem3}. For convenience we include few calculations for small values of $n$ below.

\begin{figure}[!htb]
\begin{center}
    \scalebox{0.85}{\includegraphics{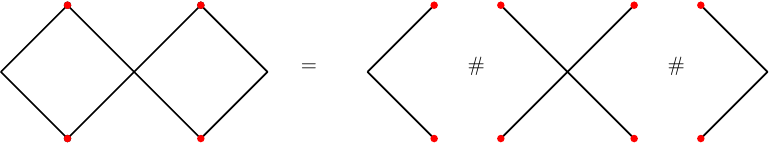}}
\caption{The graph $G_1$.}
\label{fig:n1}
\end{center}
\end{figure}

\noindent For $n=1:$ (see Figure \ref{fig:n1})
\begin{align*}
    X (X \otimes Y + Y \otimes X)X &= (X^{2} \otimes Y + XY \otimes X)X  \\ 
                                   &= X^{2} \otimes YX + XY \otimes X^2 \\&= 2Z \otimes X + 2X \otimes Z.
\end{align*}

\begin{figure}[!htb]
\begin{center}
    \scalebox{0.85}{\includegraphics{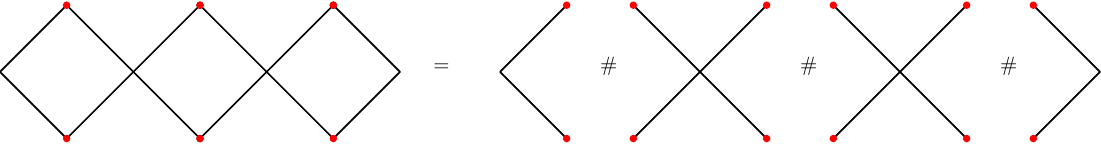}}
\caption{The graph $G_2$.}
\label{fig:n2}
\end{center}
\end{figure}

\noindent For $n=2:$ (see Figure \ref{fig:n2})
\begin{align*}
    X (X \otimes Y + Y \otimes X) ^{2} X &= X (X \otimes Y + Y \otimes X)(X \otimes Y + Y \otimes X)X \\ 
     &=  X (X \otimes XY \otimes Y + X \otimes Y^{2} \otimes X + Y \otimes X^2 \otimes Y + Y \otimes XY \otimes X) X \\ 
     &= X^{2} \otimes X \otimes YX + X^{2} \otimes Y^{2} \otimes X^2 + XY \otimes X^{2} \otimes YX + XY \otimes X \otimes X^{2} \\
     &= 2Z \otimes X \otimes X + 4 Z \otimes Y \otimes Z + 2 X \otimes Z \otimes X + 2X \otimes X \otimes Z.
\end{align*}

\begin{figure}[!htb]
\begin{center}
    \scalebox{0.85}{\includegraphics{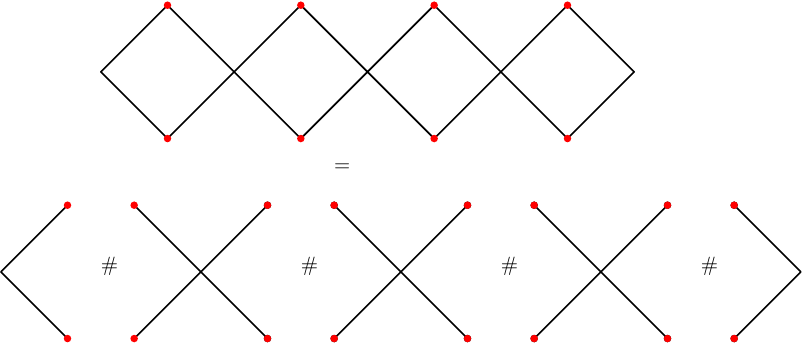}}
\caption{The graph $G_3$.}
\label{fig:n3}
\end{center}
\end{figure}

For $n=3:$ (see Figure \ref{fig:n3})
\begin{align*}
   X (X \otimes Y + Y \otimes X) ^{3} X &=  X(X \otimes Y + Y \otimes X)(X \otimes Y + Y \otimes X)(X \otimes Y + Y \otimes X) X\\
   &=X (X \otimes XY \otimes Y + X \otimes Y^{2} \otimes X + Y \otimes X^2 \otimes Y + Y \otimes XY \otimes X)\\
   & \quad \cdot (X \otimes Y + Y \otimes X) X \\
   &= X(X \otimes XY \otimes YX \otimes Y + X \otimes Y^2 \otimes X^2 \otimes Y + Y \otimes X^2 \otimes YX \otimes Y\\& \quad  + Y \otimes XY \otimes X^2  
    \otimes Y+ X \otimes XY \otimes Y^2 \otimes X + X \otimes Y^2 \otimes XY \otimes X \\ & \quad + Y \otimes X^2 \otimes Y^2 \otimes X + Y \otimes XY \otimes XY \otimes X) X \\
    &= X^2 \otimes XY \otimes YX \otimes YX + X^2 \otimes Y^2 \otimes X^2 \otimes YX + XY \otimes X^2 \otimes YX \otimes YX \\ & \quad + XY \otimes XY \otimes X^2  
    \otimes YX  + X^2 \otimes XY \otimes Y^2 \otimes X^2 + X^2 \otimes Y^2 \otimes XY \otimes X^2\\ &\quad  + XY \otimes X^2 \otimes Y^2 \otimes X^2 + XY \otimes XY \otimes XY \otimes X^2\\
    &=2Z \otimes X \otimes X \otimes X + 4Z \otimes Y \otimes Z \otimes X + 2X \otimes Z \otimes X \otimes X\\ & \quad  + 2X \otimes X \otimes Z  
    \otimes X+ 4Z \otimes X \otimes Y \otimes Z + 4Z \otimes Y \otimes X \otimes Z\\ &\quad  + 4X \otimes Z \otimes Y \otimes Z + 2 X \otimes X \otimes X \otimes Z.
\end{align*}

Based on these computations we make the following observations: 
\begin{enumerate}
    \item Each term must have at least one $Z$,
    \item Each term either starts with an $X$ or with a $Z$,
    \item By symmetry of the graph each term ends with either an $X$ or a $Z$,
    \item In general, each term starts with $ \underbrace{X\otimes \cdots \otimes X}_{k \, \, \text{times}} \otimes Z \otimes \cdots $.; by symmetry, each term ends with $\cdots \otimes Z \otimes \underbrace{X\otimes \cdots \otimes X}_{\ell \, \, \text{times}}$, for some $k, \ell \geq 0$,
    \item The positions of $Z$ and $Y$ alternates in the sense that after every $Z$ there can never be another $Z$ before another $Y$ appears,
    \item The coefficient of each term is given by $2^{\#(Z)}$, where $\#(P)$ denotes the number of times the variable $P$ occurs in the term, and
    \item Any arrangement of $X$'s,$Y$'s and $Z$'s satisfying conditions (1) to (6) must appear in the expansion of $X(X \otimes Y + Y \otimes X)^{n}X$.
\end{enumerate}
We prove these observations in the next series of lemmas. The alert reader can quickly recognise that by substituting $X \leftrightarrow 0, Z \leftrightarrow 1$ and $Y \leftrightarrow -1$ and by observations (1), (2), (3) and (4), each term defines a row of an alternating sign matrix. We will discuss this in detail soon.

\begin{lemma}\label{lem5}
    Each term in the expansion of $X(X \otimes Y + Y \otimes X)^{n}X$ must have at least one $Z$. 
\end{lemma}
\begin{proof}
    For $n=1,2$ and $3$ we see that the lemma is true.  For $n \geq 4$ we use the following expansion
\begin{align*}
    X(X \otimes Y + Y \otimes X)^{n}X &= X(X \otimes Y + Y \otimes X) (X\otimes Y + Y \otimes X )^{n-1} X \\
    &=(2Z \otimes Y + X \otimes X) (X \otimes Y + Y \otimes X)^{n-1} X \\
    &=2Z \otimes Y(X \otimes Y + Y \otimes X)^{n-1}X + X \otimes X(X \otimes Y + Y \otimes X)^{n-1}X.
\end{align*}
Now a simple induction on $n$ completes the proof of the lemma.    
\end{proof}
\begin{lemma}\label{lem6}
    Each term in the expansion $X(X \otimes Y + Y \otimes X)^{n}X$  either starts with  an $X$ or with a $Z$.
\end{lemma}
\begin{proof}
    The lemma is certainly true $n=1,2$ and $3$. For $n\geq 4$, we use the same trick like the previous lemma
    \[X(X \otimes Y + Y \otimes X)^{n}X= 2Z \otimes Y(X \otimes Y + Y \otimes X)^{n-1}X + X \otimes X(X \otimes Y + Y \otimes X)^{n-1}X \]

  \noindent Next we note that, in the expansion of $Y(X \otimes Y + Y \otimes X)^{n-1}X$ and $X(X \otimes Y + Y \otimes X)^{n-1}X$ , it can never interact with the first term. So, the first terms in the above expression remain $Z$ and $X$ respectively. This concludes the proof.
\end{proof}

\begin{lemma}\label{lem7}
      Each term in the expansion $X(X \otimes Y + Y \otimes X)^{n}X$  either ends with an $X$ or with a $Z$.  
\end{lemma}
\begin{proof}
    We can use the geometric reasoning to prove this lemma. We observe that the graph is symmetric with respect to rotation by $180^\circ$. Thus the terms appearing in its state sum decomposition must also be symmetric with respect to rotation by $180^\circ$. Then, an application of Lemma \ref{lem6} gives us the result. 
\end{proof}
We set some notation for the ease of writing in the sequel. Define $X(k)$ to be following expression
 \[ X(k) := \underbrace{X \otimes \cdots  \otimes X}_{ k \, \, \text{times}}.\]
\begin{lemma}\label{lem8}
For each term in the expansion $X(X \otimes Y+ Y \otimes X)^{n} X$, there exist a nonnegative number $k$ such that it  starts with $  \underbrace{X\otimes \cdots \otimes X}_{k \, \, \text{times}} \otimes Z \otimes \cdots $. By symmetry, there exists a nonnegative number $\ell$ such that the term ends with $ \cdots \otimes Z \otimes \underbrace{X\otimes \cdots \otimes X}_{\ell \, \, \text{times}}$.
\end{lemma}
\begin{proof}
    Our computation for $n=1, 2$ and $3$ shows that the lemma is true. For $n \geq 4$ we may use induction
    \[X(X \otimes Y + Y \otimes X)^{n}X= 2 Z \otimes Y(X \otimes Y + Y \otimes X)^{n-1} \cdot X + X \otimes X(X \otimes Y + Y \otimes X)^{n-1}X. \]
    For any term appearing in the first expansion: $Z \otimes Y(X \otimes Y + Y \otimes X)^{n-1} \cdot X$, we choose $k=0$. 
    For any term appearing in the second expansion: $X \otimes X(X \otimes Y + Y \otimes X)^{n-1}X$, it looks like $X \otimes W$ where $W$ is a term appearing in the in the expansion of $ X(X \otimes Y + Y \otimes X)^{n-1}X$. Now for $W$, induction hypothesis shall ensure the existence of $p \geq 0$ such that $W$ starts with $X(p) \otimes Y$. So in this case, choose $k= p+1$. By symmetry, we are done with the other part of the lemma as well. 
\end{proof}
\begin{lemma}\label{lem9}
    The position of $Z$ and $Y$  in each term of the expansion $X(X \otimes Y + Y \otimes X)^{n}X$ alternates. That is, after every occurrence of a $Z$, there must occur an $Y$ before another $Z$ can appear.
\end{lemma}
\begin{proof}
    By the previous computations we know that the lemma is true for $n=1, 2$ and $3$. We use induction on $n$. We first investigate the following expansion \[ Z \otimes (X \otimes Y + Y \otimes X) ^{n} X, \] for $n=1$ and $n=2$. 
    
    For $n=1$: 
    \begin{align*}
        Z \otimes (X \otimes Y + Y \otimes X) X &= Z \otimes (X \otimes X + 2 Y \otimes Z) \\
        &= Z \otimes X \otimes X + 2 Z \otimes Y \otimes Z
    \end{align*}
    
    For $n=2$: 
\begin{align*}
    Z \otimes (X \otimes Y + Y \otimes X)^{2} X &= Z \otimes (X \otimes X \otimes Y + X \otimes Y \otimes X + Y \otimes X \otimes X + 2Y \otimes Z \otimes Y) X \\
    &= Z \otimes X \otimes X \otimes X + 2Z \otimes X \otimes Y \otimes Z+ 2Z \otimes Y \otimes X \otimes Z \\&\quad + 2Z \otimes Y \otimes Z \otimes X.
\end{align*}
We set our \textit{induction hypothesis} as follows: for all $n \in \mathbb{N}$, the position of $Z$ and $Y$ alternates in the expansions of $Z \otimes (X \otimes Y + Y \otimes X)^{n} X$ and $X (X \otimes Y + Y \otimes X)^{n}X$.

The base cases have been verified for $n=1$ and $n=2$ in the preceding calculations. In general, using Lemma \ref{lem8} and the induction hypothesis, it follows from the following expansion: 
\begin{align*}
    Z \otimes (X \otimes Y + Y \otimes X)^{n} X &= Z \otimes X \otimes (X \otimes Y + Y \otimes X)^{n-1}X + Z \otimes Y \otimes X(X \otimes Y + Y \otimes X)^{n-1}X, \\ 
    X (X \otimes Y + Y \otimes X)^{n}X &= 2Z \otimes (X \otimes Y + Y \otimes X)^{n-1} X + X \otimes X(X \otimes Y + Y \otimes X)^{n-1}X.
\end{align*}
This completes the proof.
\end{proof}
\begin{lemma}\label{lem10}
    The coefficient of each term in $X(X \otimes Y + Y \otimes X)^{n}X$ is given by $2^{\#(Z)}$.
\end{lemma}

\begin{proof}
We shall deploy a similar strategy like we used for Lemma \ref{lem8}. In order to prove this lemma we use the following \textit{induction hypothesis}: the coefficient of each term appearing in the expansions of $X(X \otimes Y + Y \otimes X)^nX$ and $2Z \otimes (X \otimes Y + Y \otimes X)^{n}X$ is equal to $2^{\#(Z)}$.  

The base cases have been verified in the previous lemma. The general argument follows from the previous computation: 
\begin{align*}
    2Z \otimes (X \otimes Y + Y \otimes X)^{n} X &= 2Z \otimes X \otimes (X \otimes Y + Y \otimes X)^{n-1}X + 2Z \otimes Y \otimes X(X \otimes Y + Y \otimes X)^{n-1}X \\ 
    X (X \otimes Y + Y \otimes X)^{n}X &= 2Z \otimes (X \otimes Y + Y \otimes X)^{n-1} X + X \otimes X(X \otimes Y + Y \otimes X)^{n-1}X.
\end{align*}
To complete the argument, we note that none of the terms in $Z \otimes X \otimes (X \otimes Y + Y \otimes X)^{n-1}X$ and $Z \otimes Y \otimes X(X \otimes Y + Y \otimes X)^{n-1}X$ mix with each other because the second term is different. Similarly, the terms in $Z \otimes (X \otimes Y + Y \otimes X)^{n-1} X$ and $X \otimes X(X \otimes Y + Y \otimes X)^{n-1}X$ do not mix with each other as the first term is different. 
\end{proof}
\begin{lemma}\label{lem11}
    Any arrangement of $X$'s, $Y$'s, and $Z$'s satisfying conditions (1) to (6) from before must appear in the expansion of $X(X \otimes Y + Y \otimes X)^{n}X$.
\end{lemma}
\begin{proof}
 Let, $W=X_1 \otimes \cdots \otimes X_{n}$ be an expression where $X_i \in \{ X,Y,Z\}$. Then we define the length of $W$ to be $\ell(W) := n-1$. Then for each term $W$ appearing in the expansion of $X(X \otimes Y + Y \otimes X)^{n}X$ has $\ell(W)=n$.

 Assume $W$ is a word satisfying the conditions (1) to (5). By condition (2) we need to consider two cases separately: 
 \[ W = X \otimes W_{1}, \, \, \, \, \text{or} \, \, \, \, W= Z \otimes W_{1}.\] 
 where $\ell(W_1)=n-1$.
 
 In the first case, we note that the expression $W_1$ also satisfies the conditions (1) to (5). Then a simple induction on the $\ell(W)$ implies that $W= X \otimes W_{1}$ appears in: 
 \[ X(X \otimes Y + Y \otimes X)^{n} X = 2Z \otimes (X \otimes Y + Y \otimes X)^{n-1}X + X \otimes X (X \otimes Y + Y \otimes X)^{n-1}X \] the second expansion $X \otimes X (X \otimes Y + Y \otimes X)^{n-1}X$.

 In the second case, we can express $W$ in the following form 
 \[ W= Z \otimes X(k_1) \otimes Y \otimes X(k_2) \otimes Z \otimes \cdots  \otimes Y \otimes X(k_{l-1})\otimes Z\otimes X(k_p),\] 
 for some non-negative numbers $k_1,k_2, \cdots ,k_{p}$ and $p$ satisfying 
 \[ k_1+ \cdots + k_{p}+p=n+1\] To show the existence of $W$ in the expansion of $X(X \otimes Y + Y \otimes X)^{n} X$ we look at the following expression
\begin{align*}
    (X.X) \otimes \underbrace{(YX) \otimes \cdots \otimes (YX) }_{k_{1} \, \, \text{times}} \otimes (YY) \otimes \underbrace{(XY) \otimes \cdots \otimes (XY)}_{k_{2} \, \, \text{times}} \otimes (XX) \\
    \otimes \cdots \otimes (Y.Y) \otimes \underbrace{(X.Y) \otimes \cdots \otimes (XY)}_{k_{p-1}  \, \, \text{times}} \otimes (XX) \otimes \underbrace{(YX) \otimes \cdots \otimes (YX) }_{k_{p} \, \, \text{times}}.
\end{align*}

\noindent We note that for each tensor product $\otimes$ the two adjacent terms are $X$ and $Y$ or their reverse order. Also the expression starts and ends with an $X$. 
Hence, we can conclude that the expression must appear in $X(X \otimes Y + Y \otimes X)^{n} X$. Using the relationships between $X,Y$ and $Y$, we notice that, this is exactly equal to $ 2^{\#(Z)} \cdot W $. This concludes our proof.
\end{proof}
\noindent We can use our understanding of $v_{G_n}$, where $G_n$ is the graph in Figure \ref{fig:pp8-1} to say something about the perfect matchings of the Aztec diamond $AD(n)$. For ease we show $\ad(2)$ and $\ad(3)$ with the distinguished vertices marked in red dots in Figure \ref{fig:pp12-1}.

\begin{figure}[!htb]
\begin{center}
    \scalebox{0.85}{\includegraphics{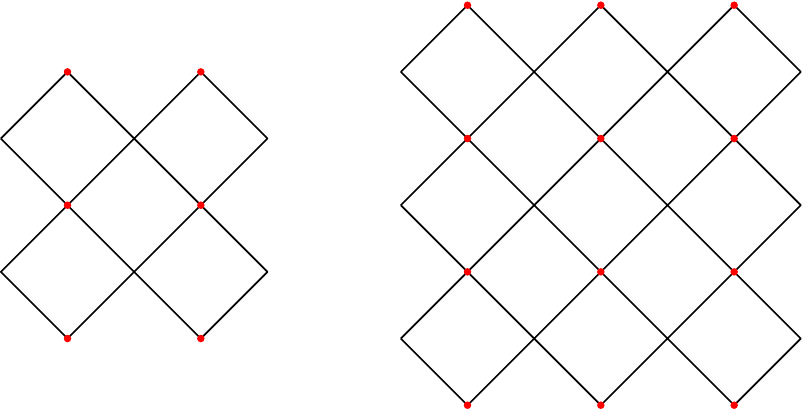}}
\caption{The graph $\ad(2)$ (left) and $\ad(3)$ (right) with the distinguished vertices marked.}
\label{fig:pp12-1}
\end{center}
\end{figure}

We can represent the state sum decomposition of $AD(n)$, $v_{AD(n)}$ as a formal sum of $n \times (n+1)$ matrices with entries in $ \{ y,n \}$. The number $M(AD(n))$ is encoded by the coefficient of the matrix where each entry is $n$, signifying a valid perfect matching. 

\begin{figure}[!htb]
\begin{center}
    \scalebox{0.85}{\includegraphics{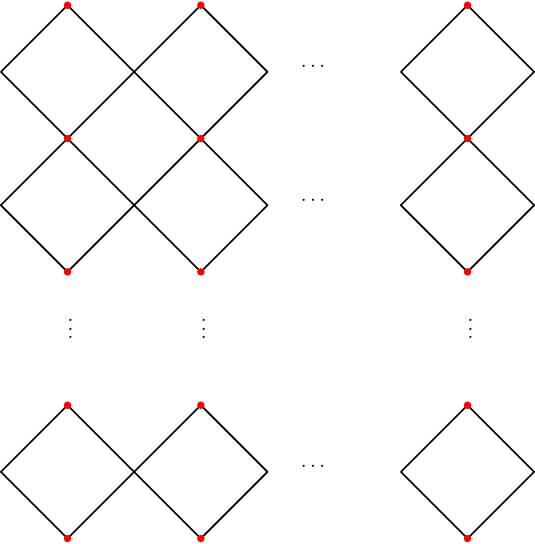}}
\caption{The graph $\ad(n)$ with the distinguished vertices marked.}
\label{fig:pp12-2}
\end{center}
\end{figure}

We now find the column-wise conditions on $X$'s, $Y$'s, and $Z$'s that ensures the presence of the vector \[ \begin{bmatrix}  n& n & \cdots & n & n
\end{bmatrix}^{t}.\] For the convenience of writing we shall investigate the row-wise conditions on $X^{t}, Y^{t}$ and $Z^{t}$ that produces the row vector \[
\begin{bmatrix}
    n& n & \cdots & n & n
\end{bmatrix}_{1 \times (n+1)}
\]
\noindent We set up our notation first. We shall use $X[k]$ to denote the following expression: \[ X[k] := \underbrace{X^{t} \cdot X^{t} \cdots X^{t} \cdot X^{t}}_{k \, \, \text{times}}  ,\]
that is, $X[k]$ is a formal sum of $(k+1) \times 1$ row vectors with entries in $\{ y, n \}$ and the operation $\cdot$ denotes the internal multiplication with the second entry of the first term with the first entry of the second term. For instance,
\begin{align*}
    X[1] &= \begin{bmatrix} y & n \end{bmatrix} + \begin{bmatrix} n & y \end{bmatrix}, \\ 
    X[2] &= (\begin{bmatrix} y & n \end{bmatrix} + \begin{bmatrix} n & y \end{bmatrix}) \cdot (\begin{bmatrix} y & n \end{bmatrix} + \begin{bmatrix} n & y \end{bmatrix}) \\ 
    &= \begin{bmatrix}
        y & (n \cdot y) & n
    \end{bmatrix} + \begin{bmatrix}
        y & (n \cdot n) & y
    \end{bmatrix} + \begin{bmatrix}
        n & (y \cdot y) & n
    \end{bmatrix}+ \begin{bmatrix}
        n & (y \cdot n) & y
    \end{bmatrix} \\
    &= \begin{bmatrix}
        y & y & n
    \end{bmatrix} + \begin{bmatrix}
        n & y & n
    \end{bmatrix} + \begin{bmatrix}
        n & y & y
    \end{bmatrix} \, \, (\text{the second term vanishes as} \, \, n \cdot n= n^2=0).
\end{align*}


\begin{lemma}\label{lem12}
The row vector $ \begin{bmatrix}  n& n & \cdots & n & n
\end{bmatrix}_{1 \times (n+1)}$  appears if and only if the following conditions on $\{X^{t} , Y^{t}, Z^{t} \}$ are satisfied
\begin{enumerate}
    \item Each term must have at least one $Z^{t}$,
    \item Each term starts with $X^{t}$ or $Z^{t}$,
    \item Each term ends with $X^{t}$ or $Z^{t}$,
    \item In general, for each term there exists a $k \geq 0$ such that it must start with $X[k] \otimes Z^{t}$. By symmetry there exists a $\ell \geq 0$ such that it end with $Z^{t} \otimes X[\ell]$, and 
    \item $Z^{t}$ and $Y^{t}$ occur alternatively.
\end{enumerate}
If the positions of the $X^t, Y^t$ and $Z^t$ satisfies the aforementioned conditions then the coefficient of \[ \begin{bmatrix}  n& n & \cdots & n & n
\end{bmatrix}_{1 \times (n+1)}\] is given by $1$. 
\end{lemma}
\begin{proof}
 Any expression $W$ of length $n$ satisfying the conditions (1) to (5) can be expressed in the following form
\[ W= X[k_1] \cdot Z^{t} \cdot X[k_2] \cdot Y^{t} \cdot X[k_3] \cdot Z^{t} \cdots Y^{t} \cdot X[k_{l-1}] \cdot Z^{t} \cdot X[k_{k_{l}}] \]
We first prove the ``if" part of the statement:
if $W$ is a general expression as above then we look at the following term in $W$ 
\begin{multline}\label{eq00}
    \underbrace{\begin{bmatrix} n & y \end{bmatrix} \cdots \begin{bmatrix} n & y \end{bmatrix}}_{k_{1} \, \, \text{times}} \cdot \begin{bmatrix} n & n \end{bmatrix} \cdot \underbrace{\begin{bmatrix} y & n \end{bmatrix} \cdots \begin{bmatrix} y & n \end{bmatrix}}_{k_2 \, \, \text{times}} \cdot \begin{bmatrix} y & y \end{bmatrix} \cdot \underbrace{\begin{bmatrix} n & y \end{bmatrix} \cdots \begin{bmatrix} n & y \end{bmatrix}}_{k_{3} \, \, \text{times}} \cdot  \\ \begin{bmatrix} n & n \end{bmatrix} \cdots \begin{bmatrix} y & y \end{bmatrix} \cdot \underbrace{\begin{bmatrix} n & y \end{bmatrix} \cdots \begin{bmatrix} n & y \end{bmatrix}}_{k_{l-1} \, \, \text{times}} \cdot \begin{bmatrix} n & n \end{bmatrix} \cdot \underbrace{\begin{bmatrix} y & n \end{bmatrix} \cdots \begin{bmatrix} y & n \end{bmatrix}}_{k_l \, \, \text{times}}
\end{multline} 
Now by virtue of the relation \[ yn =ny = n,\] in the matching algebra $\mathcal{M}$ we conclude the expression is equal to \[
\begin{bmatrix}
    n& n & \cdots & n & n
\end{bmatrix}_{1 \times (n+1)}.\]

\noindent Conversely, let $W$ be an expression with $ \{ X^{t}, Y^{t} , Z^{t} \} $ of length $n$ that produces the row vector \[\begin{bmatrix}
    n& n & \cdots & n & n
\end{bmatrix}_{1 \times (n+1)}.\]
In order to produce the row vector $\begin{bmatrix}
    n& n & \cdots & n & n
\end{bmatrix}_{1 \times (n+1)}$ out of $ \{ X^{t}, Y^{t} , Z^{t} \} $, we can not start or end with $Y^t= \begin{bmatrix}
    y & y
\end{bmatrix}$. This is simply because we can not get rid of the initial $y$ in the former case or the final $y$ in the later case. Thus each term must start with $X^{t}$ or $Y^{t}$. This proves the necessity of properties (2) and (3).

Next we prove the necessity of conditions (1) and (4). If $W$ already starts with $Z^{t}$ there is nothing to prove. Without loss of generality we may assume $W$ starts with $X^{t}$. In this case, the vector $\begin{bmatrix}
    y & n
\end{bmatrix}$ can not contribute as the initial $y$ can not be turned to $n$. Thus, we may assume the first term in given by $\begin{bmatrix}
    n & y
\end{bmatrix}$. For the second term, it has to start with $n$. If not then the second entry shall remain $y$ by virtue of the relation \[y \cdot y = y^{2}=y \] \noindent Hence, the second term must either be $Z^{t}$ or $X^{t}$. If it is $Z^{t}$ , there is nothing to prove in this case. Otherwise, it must be $X^{t}$. As before, the only term that shall contribute is $\begin{bmatrix}
    n & y
\end{bmatrix}$. Continuing in this way, we see that $Z^{t}$ must be used eventually to turn the final $y$ of $X^{t}$ to $n$. This concludes the argument. The same argument starting from backwards proves the second part of condition (4).  

Finally, we prove that condition (5) is also necessary. Suppose there is a $Y^{t}$ in $W$. We shall show that there must exist $Z^{t}$ prior and subsequent to $Y^{t}$ with number of $X^{t}-$s in between. That is, $W$ must look like \[ W = \left( \cdots Z^{t} \cdot X[k] \cdot Y^{t} \cdot X[\ell] \cdot Z^{t} \cdots  \right), \]
for some non-negative integers $k$ and $\ell$. As $Y^{t} = \begin{bmatrix}
    y & y
\end{bmatrix}$, in order to turn the second entry of $Y^{t}$ to $n$ through internal multiplication, the next term must start with $n$. Hence, the term next to $Y^{t}$ must be a $X^{t}$ or a $Z^{t}$. If it is $Z^{t}$ then we take $\ell=0$ and we are done. If it is $X^{t}$, then the term that must contribute is given by $\begin{bmatrix}
    n & y
\end{bmatrix}$. We follow the same argument. In order to turn the $y$ to $n$ through internal multiplication, the following term must start with $n$. Thus, it is either $Z^{t}$ or $X^{t}$. If it is $Z^{t}$, then we take $\ell=1$ and we are done. If it is $X^{t}$, then the term that must contribute would be $\begin{bmatrix}
    y & n
\end{bmatrix}$. Continuing in this way, we see after $\ell$ many $X^{t}$ we must use $Z^{t}= \begin{bmatrix}
    n & n
\end{bmatrix}$ to turn the $y$ to $n$.

Using the same argument we conclude that there must exist a non-negative integer $k$ such that, $W$ must look like $\left( \cdots Z^{t} \cdot X[k] \cdot Y^{t} \cdots  \right)$. Combining these two observations, we conclude that the condition (5) is also necessary.

Finally, the position of $X^{t}, Y^{t} $ and $Z^{t}$ satisfies the conditions (1) to (5),  then \eqref{eq00} is the only way to produce $\begin{bmatrix}
    n& n & \cdots & n & n
\end{bmatrix}_{1 \times (n+1)}$ with coefficient $1$. 
\end{proof}

\noindent By using Lemmas \ref{lem5}--\ref{lem11} and substituting $ X \leftrightarrow 0$, $Z \leftrightarrow 1$ and $Y \leftrightarrow -1 $  we essentially get the following theorem.


\begin{theorem}\label{thm:match-asm-1}
   Let $N_{+}(A)$ denotes the number of  $1$-s in the alternating sign matrix $A$ and $\mathcal{A_n}$ denotes the set of all $n \times n$ ASMs. Then we have the following relation
   \[ M(AD(n))= \sum_{A \in  \mathcal{A}_n} 2^{N_{+}(A)}.\] 
\end{theorem}
This is the same as Theorem \ref{thm:2-enum} and thus our proof of Theorem \ref{adm} is complete.

\begin{section}{Matching Algebra and ASMs}\label{sec:ad2}
    
    In Theorem \ref{thm:match-asm-1} we proved only one part of Theorem \ref{thm:2-enum}. In this section we prove the other part of Theorem \ref{thm:2-enum}, that is, the connection of $M(AD(n))$ with the $2$-enumeration of $(n+1) \times (n+1)$ ASMs.

\begin{lemma}\label{lem13}
The state sum decomposition of the inverted $Y$ shaped graph shown in Figure \ref{fig:lem13} with two distinguished vertices as indicated by the red dots in the figure is
$y \otimes y$. 
\end{lemma}

\begin{figure}[!htb]
\begin{center}
    \scalebox{0.75}{\includegraphics{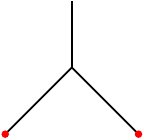}}
\caption{The $Y$ shaped graph in Lemma \ref{lem13}.}
\label{fig:lem13}
\end{center}
\end{figure}

\begin{proof}
    The graph has two distinguished corner vertices. The third corner vertex must be matched with the middle vertex. Hence, the state sum decomposition is given by $y \otimes y$.  
\end{proof}

\noindent Before we proceed further, we note that the number of perfect matchings of $\ad(n-1)$ is same as the number of perfect matching of the graph shown in Figure \ref{fig:pp15-1}, denoted by $\bd (n)$. This is because of the forced edges that must appear in any perfect matching.

\begin{figure}[!htb]
\begin{center}
    \scalebox{0.75}{\includegraphics{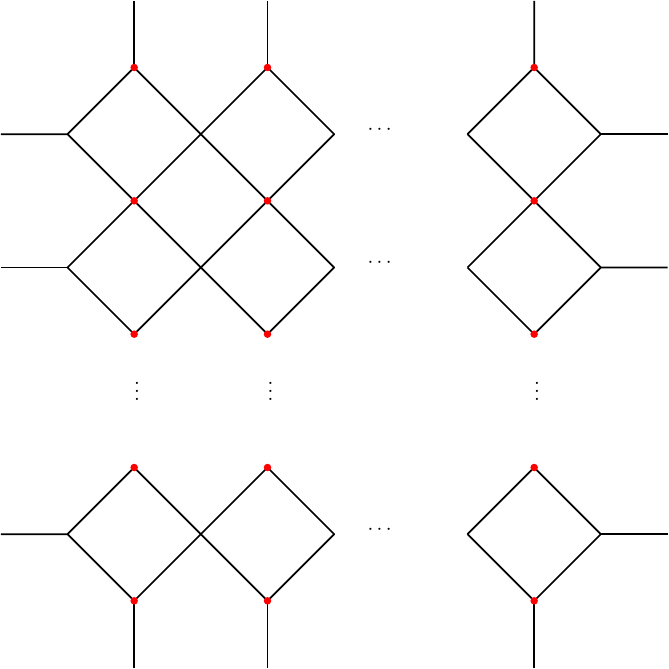}}
\caption{The graph $\bd(n)$, whose perfect matchings are equivalent to the prefect matchings of $\ad(n-1)$.}
\label{fig:pp15-1}
\end{center}
\end{figure}

We begin by calculating the state sum decompositions of the class of graphs shown in Figure \ref{fig:pp15-2}, called $\bar G_n$. To describe it concretely, it is the connected sum of $n$ $X$ shaped graphs, which are capped by the graph described in Lemma \ref{lem13} on both sides. By a repeated application of Lemmas \ref{lem3} and \ref{lem13} we see that the computation of the state sum decomposition boils down to evaluating the expression 
\[ Y (X \otimes Y + Y \otimes X) ^{n} Y .\]
We calculate the first few cases of $n$ to find some patterns, like we did in the previous section.

\begin{figure}[!htb]
\begin{center}
    \scalebox{0.75}{\includegraphics{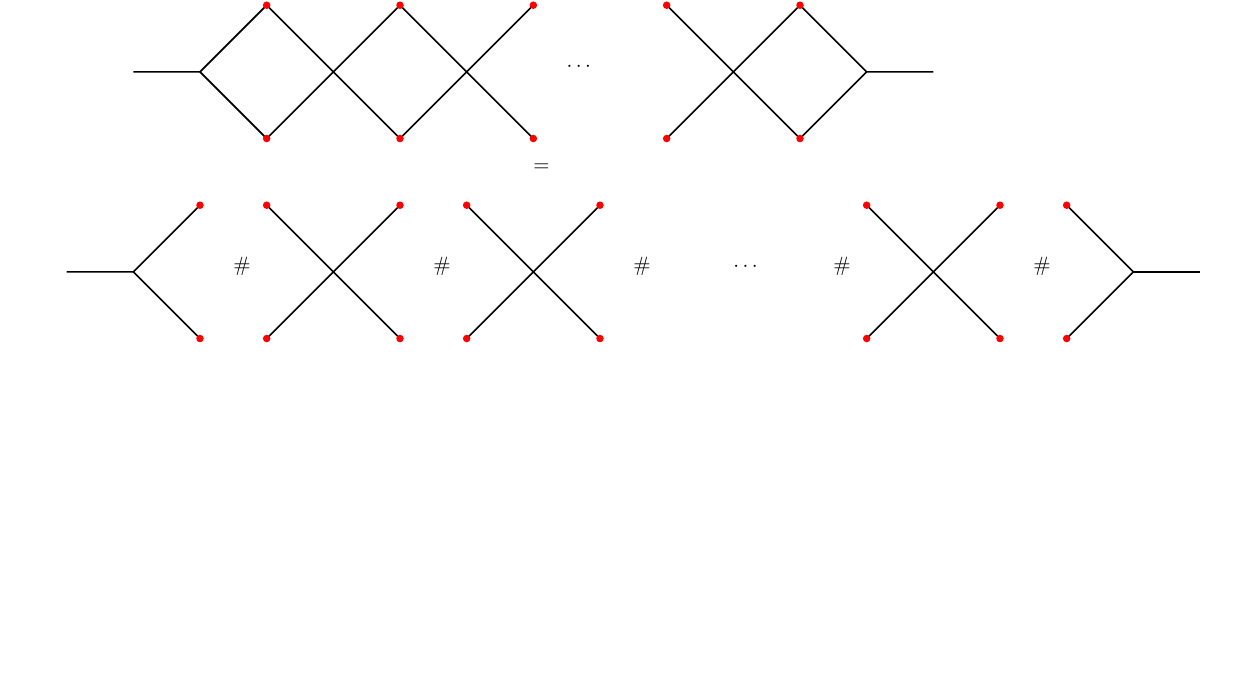}}
\caption{The graph $\bar G_n$.}
\label{fig:pp15-2}
\end{center}
\end{figure}

\noindent  For $n=1:$ (see top of Figure \ref{fig:pp15-3})
\begin{align*}
    Y (X \otimes Y + Y \otimes X)  Y= X \otimes Y + Y \otimes X.
\end{align*}

\noindent For $n=2:$ (see middle of Figure \ref{fig:pp15-3})
\begin{align*}
    Y (X \otimes Y + Y \otimes X) ^{2} Y= X \otimes X \otimes Y + X \otimes Y \otimes X + Y \otimes X \otimes X + 2 Y \otimes Z \otimes Y.
\end{align*}
\noindent For $n=3:$ (see bottom of Figure \ref{fig:pp15-3})
\begin{align*}
    Y (X \otimes Y + Y \otimes X) ^{3} Y &= X \otimes X \otimes X \otimes Y + 
    X \otimes X \otimes Y \otimes X + X \otimes Y \otimes X \otimes X \\&\quad + Y \otimes X \otimes X \otimes X  
  + 2 X \otimes Y \otimes Z \otimes Y + 2Y \otimes X \otimes Z \otimes Y\\ &\quad  + 2Y \otimes Z \otimes X \otimes Y+ 2Y \otimes Z \otimes Y \otimes X .
\end{align*}

\begin{figure}[!htb]
\begin{center}
    \scalebox{1.1}{\includegraphics{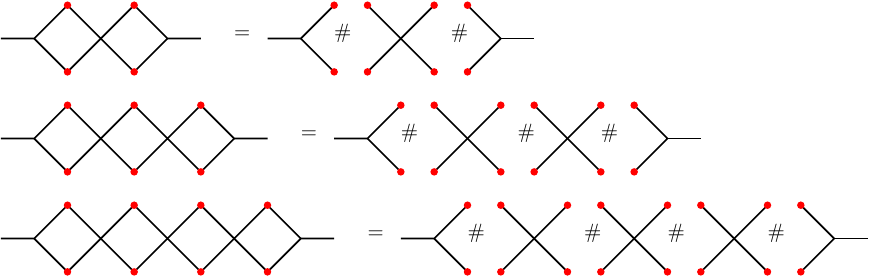}}
\caption{The graphs $\bar G_1$, $\bar G_2$ and $\bar G_3$.}
\label{fig:pp15-3}
\end{center}
\end{figure}

Based on these computations we make the following observations
\begin{enumerate}
     \item Each term must have at least one $Y$,
    \item Each term either starts with an $X$ or a $Y$, 
    \item By symmetry, each term ends with an $X$ or a $Y$.,
    \item In general, each term starts with $  \underbrace{X\otimes \cdots \otimes X}_{k \, \, \text{times}} \otimes Y \otimes \cdots $. By symmetry, each term ends with $ \cdots \otimes Y \otimes \underbrace{X\otimes \cdots \otimes X}_{\ell \, \, \text{times}} $ for some $k, \ell \geq 0$,
    \item The positions of $Y$ and $Z$ alternates, 
    \item The coefficient of each term is given by $2^{\#(Z)}$, and
    \item Any arrangement of $X$'s, $Y$'s and $Z$'s satisfying conditions (1) to (5) must appear in the expansion. 
\end{enumerate}
We shall prove these observations in the following set of lemmas. The alert reader can quickly recognise that by substituting $X \leftrightarrow 0$, $Y \leftrightarrow 1$ and $Z \leftrightarrow -1$ and by observations (1) to (6) each term defines a row of an ASM. 

\begin{lemma}\label{lem14}
    Each term in the expansion $Y (X \otimes Y + Y \otimes X) ^{n} Y$ must have at least one $Y$.
\end{lemma}
\begin{proof}
    For $n=1, 2$ and $3$, the above computation shows that the lemma is true. For $n\geq 4$, we can use the following expansion
\begin{align*}
    Y(X \otimes Y + Y \otimes X)^{n}Y &= Y(X \otimes Y + Y \otimes X)(X \otimes Y + Y \otimes X)^{n-1}Y \\ 
                                    &= (X \otimes Y + Y \otimes X)(X \otimes Y + Y\otimes X)^{n-1} Y \\
                                    &= X \otimes Y(X \otimes Y + Y \otimes X)^{n-1}Y+Y \otimes X(X \otimes Y + Y \otimes X)^{n-1}Y
\end{align*}
Now a simple induction on $n$ shall imply that each expansions in the first term contain at least one $Y$. As the second expansion already starts with $Y$, we are done. 
\end{proof}

\begin{lemma}\label{lem15}
    Each term in the expansion $Y (X \otimes Y + Y \otimes X) ^{n} Y$ starts with either $X$ or $Y$.
\end{lemma}
\begin{proof}
This is true for $n=1,2$ and $3$ as our computation shows. Now for general $n \geq 4$, this follows directly from the expansion used in the previous lemma.
\end{proof}
\begin{lemma}
    Each term in the expansion $Y (X \otimes Y + Y \otimes X) ^{n} Y$ either ends with $X$ or $Y$.
\end{lemma}
\begin{proof}
   As the graph is symmetric with respect to the rotation by $180^\circ$, its state sum expansion is also symmetric. Using Lemma \ref{lem15} completes the proof. 
\end{proof}

Before we proceed to prove the next observation, we shall make the following observation. 

\begin{lemma}
    For all $n\geq 1$, we have the following identity 
    \[ Y(X \otimes Y + Y \otimes X)^{n}Y= (X \otimes Y + Y \otimes X)^{n}.\]
\end{lemma}
\begin{proof}
The proof follows from the following series of calculations:
    \begin{align*}
       Y(X \otimes Y + Y \otimes X)^{n}Y &= Y(X \otimes Y + Y \otimes X)(X \otimes Y + Y \otimes X)^{n-1}Y \\
       &= (YX \otimes Y +  Y^{2} \otimes X)(X \otimes Y + Y \otimes X)^{n-1}Y\\
       &= (X \otimes Y + Y \otimes X) (X \otimes Y + Y \otimes X)^{n-1}Y \\
       &= (X \otimes Y + Y \otimes X)^{n} Y \\
       &= (X \otimes Y + Y \otimes X)^{n-1} (X \otimes Y + Y \otimes X)Y \\
       &= (X \otimes Y + Y \otimes X)^{n-1} (X \otimes Y^{2} + Y \otimes XY ) \\
       &= (X \otimes Y + Y \otimes X)^{n-1}(X \otimes Y + Y \otimes X) \\
       &= (X \otimes Y + Y \otimes X)^{n}.
   \end{align*}
\end{proof}
\begin{lemma}
    For each term in the expansion $(X \otimes Y + Y \otimes X)^{n}$, there exists a $k \geq 0$ so that it starts with $ \underbrace{X\otimes \cdots \otimes X}_{k \, \, \text{times}} \otimes Y \otimes \cdots $. By symmetry,  for each term, there exist a $k\geq 0$ such that it ends with $\cdots \otimes  Y \otimes \underbrace{X\otimes \cdots \otimes X}_{k \, \, \text{times}} $.
\end{lemma}
\begin{proof}
    Our computation for $n=1,2$ and $3$ shows that the lemma is true. For $n \geq 4$ we may use induction: 
    \[(X \otimes Y + Y \otimes X)^{n}= X \otimes (X \otimes Y + Y \otimes X)^{n-1}+ Y \otimes X(X \otimes Y + Y \otimes X)^{n-1}.\]
    For any terms appearing in the second expansion: $Y \otimes (X \otimes Y + Y \otimes X)^{n-1}$, we choose $k=0$. For any term appearing in the first expansion: $X \otimes (X \otimes Y + Y \otimes X)^{n-1}$, it looks like $X \otimes W$ where $W$ is a term appearing in the expansion of $(X \otimes Y + Y \otimes X)^{n-1}$. Now for $W$, induction hypothesis shall ensure the existence of $l \geq$ such that $W$ starts with $X(l) \otimes Y$. So in this case, choose $k= l+1$. By symmetry, we are done with the other part of the lemma as well.  
\end{proof}
\begin{lemma}
    The position of $Y$ and $Z$ in the expansion $Y (X \otimes Y + Y \otimes X) ^{n} Y$ alternates.
\end{lemma}
\begin{proof}
    We know the lemma is true for $n=1,2$ and $3$ by our computations above. By dint of the previous lemma, we can simply ignore the extreme $Y$'s. For any $n\geq 4$, we use induction on $n$. 
    
   We know that any term in the expansion $(X \otimes Y + Y \otimes X)^{n-1}$ can be written as $c \cdot W \otimes X$ or $c \cdot W \otimes Y$, for some constant number $c$. Here, by induction hypothesis $W$ is a term with alternating $Y$ and $Z$. In addition, in the first case: $W \otimes X $, the previous lemma guarantees that $W$ must end with $Y \otimes X(k)$. By induction hypothesis, the second case: $W \otimes Y$, $W$ must end with $Z \otimes X(k)$.
   
    In the first case: $W \otimes X$ we observe the following
    \begin{align*}
        \left( W \otimes X \right) (X \otimes Y + Y \otimes X) &= W \otimes (X \cdot X) \otimes Y + W \otimes (X \cdot Y) \otimes X \\
        &= 2W \otimes Z \otimes Y + W \otimes X \otimes X.
    \end{align*}
    In the second case: $W \otimes Y$ we observe the following
    \begin{align*}
        \left(W \otimes Y\right)\cdot(X\otimes Y+Y \otimes X) &= W \otimes (Y \cdot X) \otimes Y + W \otimes (Y \cdot Y) \otimes X \\
        &= W \otimes X \otimes Y + W \otimes Y \otimes X.
     \end{align*}
  This concludes our induction step.  
\end{proof}
\begin{lemma}
    The coefficient of each term in the expansion $Y (X \otimes Y + Y \otimes X) ^{n} Y$ is given by $2^{\#(Z)}$. 
\end{lemma}
\begin{proof}
For $n=1,2$ and $3$, our computations show the lemma is true. Choose an arbitrary term $W$ in $(X \otimes Y + Y \otimes X)^{n-1}$. There are two cases to consider: $W= W_{1} \otimes X$ or $W= W_{1} \otimes Y$. By induction hypothesis, coefficient of $W$ is given by: 
    \[ 2^{ \#(Z) \, \, \text{in} \, \, W}= 2^{\#(Z) \, \, \text{in} \, \, W_1}.\]
    In the first case
    \begin{align*}
       W(X \otimes Y + Y \otimes X) &= \left( W_1 \otimes X \right) (X \otimes Y + Y \otimes X)\\ 
       &=2W_1 \otimes Z \otimes Y + W_1 \otimes X \otimes X
    \end{align*}
    The coefficient of $W_1 \otimes Z \otimes Y = 2 \cdot 2^{\#(Z) \, \, \text{in} \, \, W_1}= 2^{\#(Z) \, \, \text{in} \, \, W_1+1}=2^{\#(Z) \, \, \text{in} \, \, W_1 \otimes Z \otimes Y} $. Similarly, the coefficient of $W_1 \otimes X \otimes X= 2^{\#(Z) \, \, \text{in} \, \, W_1}=2^{\#(Z) \, \, \text{in} \, \, W_1\otimes X \otimes X}$.
    We have a similar argument for the second case:$W= W_1 \otimes Y$ as well.
    
    Finally, we note that none of the terms mix with each other. This is simply because in the first case, each terms must end with $Z \otimes Y$ and $X \otimes X$. Also, in the second case, each terms must with $X \otimes Y$ and $Y \otimes X$. This concludes our argument. 
\end{proof}
\begin{lemma}
    Any arrangement of $X$'s, $Y$'s and $Z$'s satisfying conditions (1) to (6) of length $n$ must appear in the expansion of $Y (X \otimes Y + Y \otimes X) ^{n} Y$.
\end{lemma}
\begin{proof}
    An arrangement $W$ of $X$'s, $Y$'s and $Z$'s satisfying the conditions (1) to (6) must looks like
    \[ X(k_1) \otimes Y \otimes X(k_2) \otimes Z \otimes X(k_3)\otimes Y \otimes \cdots \otimes Z\otimes X(k_{l-1})\otimes Y \otimes X(k_{l}). \]
    We look at the following term in the expansion $Y(X \otimes Y + Y \otimes X)^{n}Y$
    \begin{align*}
        \underbrace{(YX)\otimes \cdots \otimes (YX)}_{k_1 \, \, \text{times}} \otimes (Y\cdot Y) \otimes \underbrace{(XY) \otimes \cdots \otimes (XY)}_{k_2 \,\, \text{times}} \otimes (X \cdot X) \otimes \underbrace{(YX) \otimes \cdots (YX)}_{k_3 \, \, \text{times}} \\\otimes 
        (Y \cdot Y) \otimes \cdots \otimes (X \cdot X) \otimes \underbrace{(YX) \otimes \cdots \otimes (YX)}_{k_{l-1} \, \, \text{times}} \otimes (Y \cdot Y) \otimes \underbrace{(XY) \otimes  \cdots \otimes (XY)}_{k_{l}\, \, \text{times}}.
    \end{align*}
    Clearly, the expression above appears in the expansion of $Y(X \otimes Y + Y \otimes X)^{n}Y$ and this is exactly equal to the constant times given term $W$. This completes the proof.
\end{proof}

\noindent As before, we can represent the state sum decomposition of $\bd(n)$, $v_{\bd(n)}$ as a formal sum of $n \times (n+1)$ matrices with entries in $\{ y,n\}$. By forcing of edges in the matchings, we have the following equality 
\[ M(\ad(n-1))= M (\bd(n)).\]
We shall find the coefficient of the matrix whose all entries are given by $n$. We next find the column-wise condition on $\{ X, Y, Z \}$ that ensures the presence of the vector $ \begin{bmatrix}
    y & n & \cdots & n & y \\
\end{bmatrix}^{t}$.  Note that in the column-wise condition the first and the last entry must be $y$ and rest of it must be $n$. This is because there is a forced matching vertically for each boundary vertices. We shall analyse the condition row-wise for ease of writing. We make the following observations.
\begin{lemma}\label{lem22}
The row vector $ \begin{bmatrix}
    y & n & \cdots & n & y \\
\end{bmatrix}_{1 \times (n+1)}$ appears if and only if the following conditions on $\{X^t, Y^t,Z^t \}$ are satisfied: 
\begin{enumerate}
    \item Each term must have at least one $Y^t$. 
    \item Each term starts with $X^t$ or $Y^t$. 
    \item Each term ends with $X^t$ or $Y^t$. 
    \item In general, for each term there exists a $k \geq 0$ such that it must start with $X[k] \otimes Y$. By symmetry there exists a $k \geq 0$ such that it end with $Y \otimes X[k]$.  
    
    \item $Z$ and $Y$ occur alternatively. 
\end{enumerate}
If the position of $X, Y$ and $Z$ satisfy the aforementioned conditions then the coefficient of $ \begin{bmatrix}
    y & n & \cdots & n & y \\
\end{bmatrix}^{t}$ is given by $1$. 
\end{lemma}
\begin{proof}
     Any arrangement $W$ satisfying conditions (1) to (5) can be expressed in the following form
    \[W= X[k_1] \cdot Y^{t} \cdot X[k_2] \cdot Z^{t} X[k_3] \cdot Y^{t} \cdots Z^{t} \cdot X[k_{l-1}] \cdot Y^{t} \cdot X[k_{l}]. \]
    We first prove the ``if" part of the lemma. If $W$ is a general expression as above then we look at the following term of $W$ 
    \begin{multline}
    \underbrace{\begin{bmatrix} y & n \end{bmatrix} \cdots \begin{bmatrix} y & n \end{bmatrix}}_{k_{1} \, \, \text{times}} \cdot \begin{bmatrix} y & y \end{bmatrix} \cdot \underbrace{\begin{bmatrix} n & y \end{bmatrix} \cdots \begin{bmatrix} n & y \end{bmatrix}}_{k_2 \, \, \text{times}} \cdot \begin{bmatrix} n & n \end{bmatrix} \cdot \underbrace{\begin{bmatrix} y & n \end{bmatrix} \cdots \begin{bmatrix} y & n \end{bmatrix}}_{k_{3} \, \, \text{times}} \cdot  \\ \begin{bmatrix} y & y \end{bmatrix} \cdots \begin{bmatrix} n & n \end{bmatrix} \cdot \underbrace{\begin{bmatrix} y & n \end{bmatrix} \cdots \begin{bmatrix} y & n \end{bmatrix}}_{k_{l-1} \, \, \text{times}} \cdot \begin{bmatrix} y & y \end{bmatrix} \cdot \underbrace{\begin{bmatrix} n & y \end{bmatrix} \cdots \begin{bmatrix} n & y \end{bmatrix}}_{k_l \, \, \text{times}}.
\end{multline} \label{**}
Now by virtue of the relation 
\( yn =ny=n\) in the matching algebra $\mathcal{M}$ we conclude the expression is equal to \[ \begin{bmatrix}
    y & n & \cdots & n & y
\end{bmatrix}_{1 \times (n+1)}.\]
 Conversely, let $W$ be an expression with $\{X^{t}, Y^{t}, Z^{t} \}$ of length $n$ that produces the row vector 
 \[ \begin{bmatrix}
    y & n & \cdots & n & y
\end{bmatrix}_{1 \times (n+1)}.\]
In order to produce the row vector \[ \begin{bmatrix}
    y & n & \cdots & n & y
\end{bmatrix}_{1 \times (n+1)}.\] out of $\{X^{t}, Y^{t}, Z^{t} \}$, we can not start or end with $Z^{t}$. This is simply because we can not get rid of initial $n$ in the formal case or the final $n$ in the later case through internal multiplication. Thus, each term must start or end with $X^{t}$ or $Y^{t}$. This proves the necessity of properties (2) and (3).

Next we prove the necessity of conditions (1) and (4).  If $W$ already start with $Y^{t}$ there is nothing to prove. Without loss of generality we may assume $W$ start with $X^{t}$. In this case, the vector $\begin{bmatrix}
    n & y
\end{bmatrix}$ can not contribute as the initial $n$ can not be turned to $y$ through internal multiplication. Thus, we may assume the first term is given by $\begin{bmatrix}
    y & n
\end{bmatrix}$. For the second term, it has to start with $y$. If not, then the second entry will collapse the $n$ by virtue of the relation 
\( n \cdot n = n^2 =0\) in the matching algebra $\mathcal{M}$.
Hence, the second term must either be $Y^{t}$ or $X^{t}$. If it is $Y^{t}$, there is nothing to prove in this case. Otherwise, it must be $X^{t}$. As before the only term that shall contribute is $\begin{bmatrix}
    y & n
\end{bmatrix}$. Continuing in this way, we see that $Y^{t}$ must be used eventually to turn the final $n$ to $y$. This concludes the argument. 

A similar argument from backward implies the other part of (4).

Finally, we prove the condition (5) is also necessary. Suppose there is a $Z^{t}$ in $W$. We shall show that there must exist $Y^{t}$ prior and subsequent to $Z^{t}$ with number of $X^{t}$-s in between i.e. $W$ must look like 
\[ W= \left( \cdots Y^{t}\cdot X[k] \cdot Z^{t} \cdot X[l] \cdot Y^{t} \cdots \right),\]
for some non-negative integers $k$ and $l$. As $Z^{t}= \begin{bmatrix}
    n & n
\end{bmatrix}$, in order to turn the second entry of $Z^{t}$ to $n$ through internal multiplication, the next term must start with $y$. Hence, the term next to $Z^{t}$ must be a $Y^{t}$ or a $X^{t}$. If it is $Z^{t}$ then we take $l=0$ and we are done. If it is $X^{t}$, then the term that must contribute is given by $\begin{bmatrix}
    y & n
\end{bmatrix}$. We follow the same argument. In order to turn the second entry $n$ of $X^{t}$ to $n$ through internal multiplication, the following term must start with $y$. Thus, it is either $Y^{t}$ or $X^{t}$. If it is $Y^{t}$, then we take $l=1$ and we are done. If it is $X^{t}$, then the term that contribute would be $\begin{bmatrix}
    y & n
\end{bmatrix}$. Continuing in this way, we see that after $l$ many $X^{t}$ we must use $Y^{t}= \begin{bmatrix}
    y & y
\end{bmatrix}$. 
Using the same argument backward we conclude that there must exist a non-negative integer $k$ such that, $W$ must look like $\left( \cdots Y^{t} \cdot X[k]Z^{t} \cdots \right)$. Combining the two, we conclude that the condition (5) is also necessary.

Finally, if the position $X^{t}, Y^{t}$ and $Z^{t}$ satisfies the conditions (1) to (5), then \eqref{**} is the only way to produce 
$ \begin{bmatrix}
    y & n & \cdots & n & y
\end{bmatrix}_{1 \times (n+1)}$ with coefficient $1$.
\end{proof}

\noindent Using the Lemma \ref{lem14}--\ref{lem22} and by substituting $X \leftrightarrow 0 , Y \leftrightarrow 1$ and $Z \leftrightarrow -1$, we conclude the following theorem.

\begin{theorem}
    Let $N_{-}(A)$ denotes the number of $-1$ in the alternating sign matrix $A$. Then we have the following relation 
    \[ M(\ad(n-1))= \sum_{A \in \mathcal{A}_n} 2^{N_{-}(A)}. \]
\end{theorem}
This is the other part of Theorem \ref{thm:2-enum} which now gives another proof of Theorem \ref{adm}.

\end{section}

\begin{section}{Further Applications}\label{sec:conc}

The method described by Theorem \ref{thm:main} is quite general and we can apply it to several other cases. In this paper we have only applied it to the case of an unrestricted Aztec Diamond and counted the number of perfect matchings, thereby reproving Theorem \ref{adm}. In a sequel to this paper, we plan to apply our method to two separate problems, both of which are unsolved as of now.

The symmetry classes of domino tilings of $\ad(n)$ under the action of the dihedral group of order $8$ have also been studied by Ciucu \cite{CiucuSymmetry} and Yang \cite{Yang}. So far, all the number of perfect matchings of the symmetry classes except that of a diagonally symmetric and diagonally \& anti-diagonally symmetric Aztec Diamonds have been enumerated with product formulas. In the sequel to this work we plan to examine the remaining two cases and shed some light on possible formulas to enumerate the number of perfect matchings of these missing cases.
\end{section}

\end{document}